\documentclass[11pt,a4paper,oneside,reqno]{amsart}
\usepackage{mathrsfs}
\usepackage{amsfonts}
\usepackage{amssymb}
\usepackage{amsxtra}
\usepackage{dsfont}
\usepackage{amsthm}
\usepackage{geometry}
\usepackage{float}
\usepackage{amsmath, amssymb, amsthm}
\usepackage[utf8]{inputenc}
\usepackage[T1]{fontenc}
\usepackage{graphicx}
\usepackage[unicode, psdextra]{hyperref}
\usepackage{listings}
\usepackage{indentfirst}
\usepackage{theoremref}
\usepackage{thmtools}
\usepackage{physics}
\usepackage{mathtools}
\usepackage[polish,english]{babel}
\usepackage{color}
\usepackage{bbm}

\newcommand{\pl}[1]{\foreignlanguage{polish}{#1}}

\usepackage{hyphenat}

\newcommand{\vj}{j}
\newcommand{\R}{\mathbb{R}}

\newcommand{\N}{\mathbb{N}}
\newcommand{\No}{\mathbb{N}_{\rm odd}}

\newcommand{\Q}{\mathbb{Q}}
\newcommand{\mH}{\mathcal{H}}
\newcommand{\mD}{\mathcal{D}}
\newcommand{\mF}{\mathcal{F}}
\newcommand{\mP}{\mathcal{P}}
\newcommand{\mS}{\mathcal{S}}

\definecolor{green}{rgb}{0,0.5,0}
\newcommand{\ind}[1]{\mathbbm{1}_{#1}}

\newtheorem{theorem}{Theorem}[section]
\newtheorem{lemma}[theorem]{Lemma}
\newtheorem{pro}[theorem]{Proposition}
\newtheorem{cor}[theorem]{Corollary}
\theoremstyle{remark}

\theoremstyle{definition}

\numberwithin{equation}{section}

\title[Dim-free estimates for odd order maximal Riesz transforms]{Dimension-free $L^p$ estimates for odd order maximal Riesz transforms in terms of the Riesz transforms}

\author{Maciej Kucharski}
\address{Maciej Kucharski\\
	Instytut Matematyczny\\
	Uniwersytet \pl{Wroc{\l}awski}\\
	Plac Grun\-waldzki 2\\
	50-384 \pl{Wroc{\l}aw}\\
	Poland}
\email{mkuchar@math.uni.wroc.pl}

\author{B{\l}a{\.z}ej Wr{\'o}bel}
\address{B{\l}a{\.z}ej Wr{\'o}bel\\
	Instytut Matematyczny\\
	Uniwersytet \pl{Wroc{\l}awski}\\
	Plac Grun\-waldzki 2\\
	50-384 \pl{Wroc{\l}aw}\\
	Poland}
\email{blazej.wrobel@math.uni.wroc.pl}

\author{Jacek Zienkiewicz}
\address{Jacek Zienkiewicz\\
	Instytut Matematyczny\\
	Uniwersytet \pl{Wroc{\l}awski}\\
	Plac Grun\-waldzki 2\\
	50-384 \pl{Wroc{\l}aw}\\
	Poland}
\email{Jacek.Zienkiewicz@math.uni.wroc.pl}


\thanks{Maciej Kucharski and B{\l}a{\.z}ej Wr{\'o}bel were supported by the National Science Centre (NCN), Poland  research project Preludium Bis 2019/35/O/ST1/00083. \\
\indent This article was merged into arXiv:2305.09279.}

\begin{document}

	\begin{abstract}
		We prove a dimension-free $L^p(\R^d)$, $1<p<\infty$, estimate for the vector of maximal Riesz transforms of odd order in terms of the corresponding Riesz transforms. This implies a dimension-free $L^p(\R^d)$ estimate for the vector of maximal Riesz transforms in terms of the input function. We also give explicit estimates for the dependencies of the constants on $p$ when the order is fixed.  Analogous dimension-free estimates are also obtained for single Riesz transforms of odd orders with an improved estimate of the constants. These results are a dimension-free extension of the work of J. Mateu, J. Orobitg, C. P\'erez, and J. Verdera. Our proof consists of factorization and averaging procedures, followed by a non-obvious application of the method of rotations. 
	\end{abstract}
	
	\maketitle
	\section{Introduction} 
	\label{sec:Int}

	Fix a positive integer $k$ and denote by $\mH_k$ the space of spherical harmonics of degree $k$ on $\R^d.$ Then, for $P \in \mH_k$ the Riesz transform $R=R_P$ is defined by the kernel
	\begin{equation}
		\label{eq:KP}
	K_P(x)=K(x) = \gamma_k \frac{P(x)}{\abs{x}^{k+d}} \qquad\textrm{ with } \qquad  \gamma_k = \frac{\Gamma\left( \frac{k+d}{2}\right)}{\pi^{d/2}\Gamma\left( \frac{k}{2}\right)},
\end{equation}
	more precisely,
	\begin{equation} \label{eq:R}
		R_P f(x) = \lim_{t \to 0} R_P^t f(x), \qquad\textrm{ where } \qquad R_P^t f(x) = \gamma_k \int_{\abs{y}>t} \frac{P(y)}{\abs{y}^{k+d}} f(x-y) dy.
	\end{equation}
	The operator $R_P^t$ is called the \emph{truncated Riesz transform}.  In the particular case of $k=1$ it coincides with the classical first order Riesz transforms $R_j:=R_{x_j}.$ It is well known, see \cite[p. 73]{stein}, that the Fourier multiplier associated with the Riesz transform $R_P$ equals 
	\begin{equation} \label{eq:m}
		m_P(\xi) = (-i)^k\frac{P(\xi)}{\abs{\xi}^{k}}.
	\end{equation}
	Since $P$ is homogeneous of degree $k$, by the above formula $m_P$ is bounded and %the above formula together with 
	Plancherel's theorem implies the $L^2(\R^d)$ boundedness of $R_P.$
	The $L^p(\R^d)$ boundedness of the single Riesz transforms $R_P$ for $1<p<\infty$ follows from the Calder\'on--Zygmund method of rotations \cite{CZ2}. 
	
	%	{\color{blue} When $k$ is fixed, this approach gives a dimension-free estimate for the $L^p(\R^d)$ norm of $R_P$. However, for even $P$ the estimate
%is polynomial in the dimension $d$.} 

	The systematic study of the dimension-free $L^p$ bounds
	for the Riesz transforms has begun in the seminal paper of
	E. M. Stein \cite{stein_riesz}. He has proved the $\ell^2$ vector-valued estimates
	for the vector of the first order Riesz transforms
	$(R_1f,\ldots,R_df)$. More precisely,
	\begin{equation}
		\label{eq:Riesz0}
		\norm{\bigg(\sum_{j=1}^d |R_jf|^2\bigg)^{1/2}}_{L^p(\R^d)}\leqslant C_p\, \|f\|_{L^p(\R^d)},\qquad 1<p<\infty,\end{equation}
	where $C_p$ is independent of the dimension $d.$ 
	
	This result has been extended to many other settings. The
	analogue of the dimension-free inequality \eqref{eq:Riesz0} has also been proved for higher order Riesz transforms, see \cite[Th\'{e}or\`{e}me 2]{duo_rubio}. The optimal constant $C_p$ in \eqref{eq:Riesz0} remains unknown when $d\geqslant 2;$ however  the best results to date  given in \cite{BW1} (see also \cite{DV})  established the correct order of the dependence on $p$. We note that the explicit values of $L^p(\R^d)$ norms of the single first order Riesz transforms $R_j,$ $j=1,\ldots,d,$ were obtained by Iwaniec and Martin \cite{iwaniec_martin} based on the complex method of rotations.

	In this paper we study  the relation between $R_P$ and the \emph{maximal Riesz transform} defined by
	\begin{equation*}
		%\label{eq:R*}
		R_P^* f(x) = \sup_{t > 0} \abs{R_P^t f(x)}.
	\end{equation*}
	There is an obvious pointwise inequality $R_P f(x)\leqslant R_P^*f(x).$ 
	In a series of papers \cite[Theorem 1]{mateu_verdera} (first order Riesz transforms), \cite[Section 4]{mopv} (odd order higher Riesz transforms), and \cite[Section 2]{mov1} (even order higher Riesz transforms), J. Mateu, J. Orobitg, C. P\'erez, and J. Verdera proved that also a reverse inequality holds in the $L^p(\R^d)$ norm. 
	Namely, together the results of  \cite{mateu_verdera,mopv,mov1} imply that for each $1<p<\infty$ there exists a constant $C(p,k,d)$ such that
	\begin{equation} \label{eq:mat_ver}
		\norm{R_P^* f}_{L^p(\R^d)}\leqslant C(p,k,d)\norm{R_P f}_{L^p(\R^d)}
	\end{equation}
	for all $f\in L^p(\R^d)$. 
	
	The estimate \eqref{eq:mat_ver} as presented in \cite{mateu_verdera,mopv,mov1} has been proved for general singular integral operators with
	even kernels \cite{mov1} or with odd kernels \cite{mopv}. The cost of this generality is that the values of $C(p,k,d)$ grow exponentially with the dimension. In view of \cite{Jan}, the question about improved rate arises naturally.

	%beyond higher order Riesz transforms case. The results remains true for the singular integral operators with even kernels \cite{mov1} and with odd kernels \cite{mopv}. The proof of \eqref{eq:mat_ver} presented in \cite{mateu_verdera,mopv,mov1} produces a constant $C(p,k,d)\leqslant 4^d$ of an exponential growth. %$C(p,k,d)\leqslant C(p,k)4^d$ {\color{blue} <- Do sprawdzenia. Maciek ?}.
	
	Recently, the first and the second named author proved that when $p=2$, in  \eqref{eq:mat_ver} one may take an explicit dimension-free constant $C(2,1,d)\leqslant 2\cdot 10^8,$ see \cite[Theorem 1.1]{kw}. The arguments applied in \cite{kw} relied on Fourier transform estimates together with 
	square function techniques developed by Bourgain \cite{bou1} for studying dimension-free estimates for maximal functions associated with symmetric convex bodies. Extension of this approach to other $p$ seems to be delicate due to the lack of the necessary $L^1$ behaviour of the operators $M^t$ defined in \cite[eq.\ (3.5)]{kw}. 
	
	%The lack of the immediate $L^1$ interpolation point brings some 
	%technical difficulties and seem to limit the approach of \cite{kw} to the case $p=2$.
	%However, for the problem of estimating $C(p,1,d)$ this approach only seems to work when $p=2,$ see \cite[Remarks 4,5]{kw}. 
	
	In this paper we prove that the dimension-free estimate from \cite{kw} still holds 
	for odd order Riesz transforms and for $1<p<\infty$.
	The main result of our paper is the following square function estimate of the vector of maximal Riesz transforms in terms of the Riesz transforms.
	
	%The main goal of this paper is to prove that the
	%estimate \eqref{eq:mat_ver} with the dimension free 
	%$C(p,k,d)$ holds for odd order  Riesz transforms
	%and  all $L^p(\R^d)$ spaces, $1<p<\infty$.
	
	%that for  the bound \eqref{eq:mat_ver} can be improved to a dimension-free estimate on all $L^p(\R^d)$ spaces, $1<p<\infty$. 
	
	%This can be thought of as a dimension-free improvement of \cite[Section 4]{mopv} and an extension of \cite{kw} beyond $p=2$ and to higher order Riesz transforms.
	
	%The main theorem of our paper states that we have a square function estimate on $L^p(\R^d)$ for the vector of maximal truncated Riesz transforms \eqref{eq:R*} in terms of the vector of Riesz transforms.
	\begin{theorem} \thlabel{thm1}
		Take $p \in (1, \infty)$ and let $k\leqslant d$ be a non-negative odd integer. Let $\mP_k$ be a subset of $\mH_k.$  Then there is a constant $A(p,k)$ independent of the dimension $d$ and such that
		\begin{equation*} 
			%\label{eq:thm1}
			\norm{\left(\sum_{P\in \mP_k} |R_P^* f|^2\right)^{1/2}}_{L^p(\R^d)} \leqslant A(p,k) \norm{ \left(\sum_{P\in \mP_k} |R_P f|^2\right)^{1/2}}_{L^p(\R^d)},
		\end{equation*}
		where $f\in L^p(\R^d).$ Moreover, for fixed $k$ we have $A(p,k)=O(p^{3+k/2})$ as $p\to \infty$ and $A(p,k)=O((p-1)^{-3-k/2})$ as $p\to 1.$ 
	\end{theorem}
	In particular, if $\mP_k$ contains one element $P,$ then \thref{thm1} immediately gives
	\begin{equation*}
		%\label{eq:Rs}
		\norm{ R_P^* f}_{L^p(\R^d)} \leqslant A(p,k) \norm{  R_P f}_{L^p(\R^d)}.
	\end{equation*}
	In this case however, we can slightly improve the constant $A(p,k).$
	\begin{theorem} \thlabel{thm2}
		Take $p \in (1, \infty)$ and let $k\leqslant d$ be a non-negative odd integer. Let $P$ be a spherical harmonic of degree $k.$ Then there is a constant $ B(p,k)$ independent of the dimension $d$ and such that
		\begin{equation*} 
			%\label{eq:thm2}
			\norm{ R_P^* f}_{L^p(\R^d)} \leqslant B(p,k) \norm{  R_P f}_{L^p(\R^d)},
		\end{equation*}
		where $f\in L^p(\R^d).$ Moreover, for fixed $k$ we have $B(p,k)=O(p^{2+k/2})$ as $p\to \infty$ and $B(p,k)=O((p-1)^{-2-k/2})$ as $p\to 1.$ 
	\end{theorem}
	%\begin{remark}
	%	\label{rem:thR1}
	%	Keeping track of the constants one may come up with the explicit bound $C(p,k)\leqslant ???.$ 
	%\end{remark}

	Our last main result follows from a combination of \thref{thm1} with  a result of Duoandikoetxea and Rubio de Francia \cite[Th\'eor\`eme 2]{duo_rubio}. Denote by $a(d,k)$ the dimension of $\mH_k$ and let $\{Y_j\}_{j=1,\ldots,a(d,k)}$ be an orthogonal basis of $\mH_k$ normalized by the condition
	\[
	\frac{1}{\sigma(S^{d-1})}\int_{S^{d-1}} |Y_j(\theta)|^2\,d\sigma(\theta) =\frac{1}{a(d,k)};
	\]
	here $d\sigma$ denotes the (unnormalized) spherical measure. 
	\begin{cor}
		\thlabel{cor:dr}
		Take $p \in (1, \infty)$ and let $k\leqslant d$ be a non-negative odd integer. Then there is a constant $G(p,k)$ independent of the dimension $d$ and such that
		\begin{equation*} 
			%\label{eq:th1}
			\norm{\left(\sum_{j=1}^{a(d,k)} |R_{Y_j}^* f|^2\right)^{1/2}}_{L^p(\R^d)} \leqslant G(p,k) \|f\|_{L^p(\R^d)},
		\end{equation*}
		where $f\in L^p(\R^d).$ Moreover, for fixed $k$ we have $G(p,k)=O(p^{4+k/2})$ as $p\to \infty$ and $G(p,k)=O((p-1)^{-4-k})$ as $p\to 1.$  
	\end{cor}
	
	\subsection{Structure of the paper and our methods}
There are three main ingredients used in the proofs of  \thref{thm1,thm2}. 
	
	Firstly, we need a factorization of the truncated Riesz transform $R_P^t=M^t_k(R_P)$. Here, $M^t_k,$ $t>0,$ is a family of radial Fourier multiplier operators. In the case $k=1$ this factorization has been one of the key steps in establishing the main results of \cite{kw}. In particular the operator $M_1^t$ considered here coincides with $M^t$ defined in \cite[(eq.) 3.5]{kw}. For general values of $k$ the factorization is also implicitly contained in \cite[Section 2]{mateu_verdera} ($k=1$), \cite[Section 2]{mov1} ($k$ even), and \cite[Section 4]{mopv} ($k$ odd). Note that for the first order Riesz transforms the formulas $R_j^t=M^t_1(R_j),$ $j=1,\ldots,d,$ together with the identity $I=-\sum_{j=1}^d R_j^2$ imply that 
	\begin{equation}
		\label{eq:MtR1} 
		M^t_1=-\sum_{j=1}^d M_1^t R_j^2=-\sum_{j=1}^d R_j^t R_j.\end{equation}
	Details of the factorization procedure are given in Section \ref{sec:fa}.

	The second ingredient we need is an averaging procedure. It turns out that a useful analogue of \eqref{eq:MtR1}  is not directly available for Riesz transforms of orders higher than one. The reason behind it is the fact that not all compositions of first-order Riesz transforms are higher order Riesz transforms according to our definition. For instance, in the case $k=3$  the multiplier symbol of $R_1^3=R_1 R_1 R_1$ on $L^2(\R^2)$ equals $\xi_1^3/|\xi|^3$  and $P(\xi)=\xi_1^3$ is not a spherical harmonic. However, the formula $$I=-\sum_{j_1=1}^d \sum_{j_2=1}^d \sum_{j_3=1}^d R_{j_1}^2 R_{j_2}^2 R_{j_3}^2,$$
	includes squares of all compositions of Riesz transforms including $R_1^6=(R_1^3)^2$. Therefore the above formula does not directly lead to an expression of $M^t$ in terms of $R^t_P$ and $R_P.$ To overcome this problem we average over the special orthogonal group $SO(d).$  Then we obtain
	\begin{equation} 
		\label{eq:lemAp}
		M^t_kf(x) = C(d,k)\int_{SO(d)}   \sum_{j \in I} (R_j^t R_j f)_U(x)\, d\mu(U),
	\end{equation}
	see \thref{pro:av}. Here $T_U$ is the conjugation of an operator $T$ by $U\in SO(d),$ see \eqref{conj_U_def}, $d\mu$ denotes the normalized Haar measure on $SO(d),$  while $C(d,k)$ is a constant. The symbol $I$ denotes the set of distinct indices $j=(j_1,\ldots,j_k)$ while $R_j^t$ and $R_j$ are the truncated Riesz transforms and the Riesz transforms \eqref{eq:R} corresponding to the monomials $P_j(x)=x_{j_1}\cdots x_{j_d}.$ Note that since $j\in I$ the polynomials $P_j$ are spherical harmonics and thus the operators $R_j$ are indeed higher order Riesz transforms. In view of \eqref{eq:lemAp},  if we demonstrate that that $C(d,k)$ is bounded by a universal constant, we are left with estimating the maximal function corresponding to
	$\sum_{j \in I} (R_j^t R_j).$ The reduction via the averaging procedure is described in detail in Section \ref{sec:av}. It is noteworthy that in order for the averaging approach to work it is essential that for each order $k$ the multiplier symbols of $M^t_k$ are radial functions.

	The third main ingredient of our argument is the method of rotations. We use it to estimate  the maximal function corresponding to $\sum_{j \in I} (R_j^t R_j)$ and thus the first two ingredients described above are crucial in reaching this point. In the context of dimension-free estimates for Riesz transforms this method has been first employed by Duoandikoetxea and Rubio de Francia \cite{duo_rubio}. However, a direct application of their techniques seems not well suited for our problem. Indeed, it only allows one to prove a weaker variant of \thref{cor:dr}, with the supremum taken outside of the $\ell^2$ norm, cf.\ \cite[Theorem 1.3]{kw}. In order to make the method of rotations work in our problem we need several duality arguments, Khintchine's inequality, and some specific computations. All of it reflects the size of the constants $A(p,k)$ in \thref{thm1} and $B(p,k)$ in \thref{thm2}. The application of the method of rotations to our problem is described in detail in Section \ref{sec:mr}.
	
	%{\color{blue} Wydaje mi się, że warto napisać zdanie niżej albo podobne. Jak kogoś nie interesuje zależność od $p,$ to będzie mu się łatwiej czytać}
	
	At the first reading it might be helpful to skip the explicit values of constants in terms of $k$ and $p$ and only focus on these constants being independent of the dimension $d.$ An interested reader may trace the exact dependencies of the constants in terms of $k$ and $p$ in the paper.

%	Parzystego rzędu

%{\color{blue} Akapit poniżej trzeba inaczej uargumentować jeśli chcemy powiedzieć że nawet dla lokalnych w czasie urojonych potęg jest źle, z tego co mówiłeś (Jacek) masz na to argument. Pytanie czy chcemy o tym wspominać. To co wpisałem jest śliskie, ale może chyba być we wstępie.} 
%	
%	We finish the introduction with some remarks concerning even order higher Riesz transforms. In this case both the factorization step and the averaging step seem feasible. What is missing is a method of rotations that produces suitable maximal operators. When $k$ is odd the method of rotations reduces the task to maximal functions corresponding to directional truncated Hilbert transforms. However in the case of even $k$ this method leaves us with maximal functions which correspond to 'directional' imaginary powers $\{I_{1,u}^{it}\}$ to be estimated, see \cite[p.\ 195]{duo_rubio}. This is problematic as the $L^p(\R^d)$ norm of $\{I_{1,u}^{it}\}$ seems to behave at best like $(1+|t|)^{d|1/p-1/2|}$ and clearly $\|I_{1,u}^{it}f\|_{L^p(\R^d)}\leqslant \|\sup_{s>0} |I_{1,u}^{is}f|\|_{L^p(\R^d)}.$   

\subsection{Notation}
\label{sec:not}
We finish the introduction with a description of the notation and conventions used in the rest of the paper. 

\begin{enumerate}

	\item The letters $d$ and $k$ stand for the dimension and for the order of the Riesz transforms, respectively.  In particular we always have $k\leqslant d,$ even if this is not stated explicitly.  
	
	\item The symbol $\N$ represents the set of positive integers. We write $\No$ for the set of odd elements of $\N.$ Throughout the rest of the paper we always assume that $k\in \No$.

	\item For an exponent $p\in [1,\infty]$ we let $q$ be its conjugate exponent satisfying
	$$1=\frac1p+\frac1q.$$
	 When $p\in(1,\infty)$ we set
	$$p^*:=\max(p,(p-1)^{-1}).$$

	\item We abbreviate $L^p(\R^d)$ to $L^p$ and $\norm{\cdot}_{L^p}$ to $\norm{\cdot}_p$. For a sublinear operator  $T$ on $L^p$ we denote by $\|T\|_{p\to p}$ its norm. We let $\mathcal S$ be the space of Schwartz functions on $\R^d.$ Slightly abusing the notation we say that a sublinear operator $T$  is bounded on $L^p$ if it is bounded on $\mathcal S$ in the $L^p$ norm.  For $k\in \N$ we let $\mD(k)$ be the linear span of  $\{R_P(f)\colon P\in \mH_k, f\in \mS\}.$ Since $R_P$ is bounded on $L^p$ for $1<p<\infty$   the space $\mD(k)$ is then a subspace of each of the $L^p$ spaces.
	
	\item For a Banach space $X$ the symbol $L^p(\R^d;E)$ stands for the space of weakly measurable functions  $f\colon \R^d\to E$ equipped with the norm $\|f\|_{L^p(\R^d;E)}=(\int_{\R^d}\|f(x)\|_E^p\,dx)^{1/p}.$ Similarly, for a finite set $F$ by $\ell^p(F;E)$ we denote the Banach space of $E$-valued sequences $\{f_s\}_{s\in F}$ equipped with the norm $\|f\|_{\ell^p(F;E)}=(\sum_{s\in F}\|f_s\|_E^p)^{1/p}.$
	
	\item The symbol $C_{\Delta}$ stands for a constant that possibly depends on $\Delta>0.$ We write $C$ without a subscript when the constant is universal in the sense that it may depend only on $k$ but not on the dimension $d$ nor on any other quantity.
	
	\item For two quantities $X$ and $Y$ we write $X\lesssim_{\Delta} Y$  if $X \leqslant C_{\Delta} Y$ for some constant $C_{\Delta}>0$ that depends only on $\Delta.$ We abbreviate $X\lesssim Y$ when $C$ is a universal constant. We also write $X\sim Y$ if both $X\lesssim Y$ and $Y \lesssim X$ hold simultaneously. By $X\lesssim^{\Delta} Y$ we mean that $X\leqslant C^{\Delta} Y$ with  a universal constant $C.$ Note that in this case $X^{1/\Delta}\lesssim Y^{1/\Delta}.$  
	
	\item The symbol $S^{d-1}$ stands for the $(d-1)$-dimensional unit sphere in $\R^d$ and by $\omega$ we denote the uniform measure on $S^{d-1}$ normalized by the condition $\omega(S^{d-1})=1.$ We also write 
	\begin{equation}
		\label{eq:Sd-1}
	S_{d-1} = \frac{2\pi^{d/2}}{\Gamma\left( \frac{d}{2} \right)}
	\end{equation} to denote the unnormalized surface area of $S^{d-1}.$    
	
	\item The Fourier transform is defined for $f\in L^1$ and $\xi\in\R^d$  by the formula
	\[
	\widehat{f}(\xi) = \int_{\R^d} f(x) e^{-2 \pi i x \cdot \xi} dx.
	\]
	
	\item The Gamma function is defined for $s>0$ by the formula
	\[
	\Gamma(s) = \int_0^\infty t^{s-1}e^{-t}dt.
	\]
	We shall often use Stirling's approximation for $\Gamma(s)$ 
	\begin{equation}
		\label{StirF}
	\Gamma(s)\sim\sqrt{2\pi}s^{s-\frac12}e^{-s},\qquad s\to \infty.
	\end{equation}
\end{enumerate}

\section{Factorization}

\label{sec:fa}

The first goal of this section is to show that a factorization formula for $R_P^t$ in terms of $R_P$ is feasible. Proposition below is implicit in \cite[Section 4]{mopv}. \begin{pro}
	\thlabel{pro:fact}
	Let $k\in \No$. Then there exists a family of  operators $M_k^t,$ $t>0$, 
	which are bounded on $L^p,$ $1<p<\infty,$ and such that for all $P\in \mH _k$ we have 
	\begin{equation}
		\label{eq:fact}
		R_P^t f=M_k^t (R_P f),
	\end{equation}
	where $f\in L^p.$ Each $M^t_k$ is a convolution operator with a radial convolution kernel $b^t_k.$ Moreover, when $P\in \mH_k$ and $f\in \mS,$ then for a.e.\ $x\in \R^d,$ the function $t\mapsto M_k^t (R_P f)(x)$ is continuous on $(0,\infty).$   
\end{pro}
\begin{proof}
	Let $c_d=\frac{\Gamma((d-1)/2)}{2\pi^{d/2}\Gamma(1/2)},$ $N=(k-1)/2,$ and denote by $B$ the open Euclidean ball of radius $1$ in $\R^d.$ 
	It is justified in \cite[pp.\ 3674--3675]{mopv} that the function
	\begin{equation}
		\label{eq:bdef}
			b(x)=b_{k,d}(x):=\sum_{j=1}^d R_j\left[y_j\cdot h(y)\right](x),
	\end{equation}
where
 $$h(y)=c_d(1-d)\frac{1}{|y|^{d+1}}\ind{B^c}(y)+(\beta_1 +\beta_2|y|^2+\cdots+\beta_N  |y|^{2N-2})\ind{B}(y),$$
satisfies the formula 
	\begin{equation}
		\label{eq:RKB}
		R_P(b)(x)=K_P(x)\ind{B^c}.
	\end{equation}
	Here $\beta_1,\ldots,\beta_N$ are constants which depend only on $k$ and $d$ and whose exact value is irrelevant for our considerations, and $K_P,$ $R_P$ have been defined in \eqref{eq:KP}, \eqref{eq:R}, respectively. The important point is that \eqref{eq:RKB} remains true for any $P\in \mH_k.$ 
	
Denote by $H$ the radial profile of the Fourier transform of $h$, i.e.\ $H(|\xi|)=\widehat{h}(\xi)$ for $\xi\in \R^d.$ By taking the Fourier transform of \eqref{eq:bdef} it is straightforward to see that $b$ is a radial function. This follows since the multiplier symbol of $R_j$ is $-i\xi_j/|\xi|$ and  $$\widehat {(y_j h(y))}(\xi)=\frac{\xi_j}{-2\pi i|\xi|}\,H'(|\xi|),$$ so that
	\begin{align*}
		\mF b(\xi)=\sum_{j=1}^d \frac{\xi_j^2}{2\pi |\xi|^2}\cdot H'(|\xi|)=\frac{1}{2\pi}H'(|\xi|)
	\end{align*}
	is indeed radial and so is $b.$

	Let $b^t(x)=b^t_k(x):=t^{-d}b(x/t)$ be the $L^1$ dilation of $b;$ clearly $b^t$ is still radial.  The dilation invariance of $R_P$ together with \eqref{eq:RKB} leads us to the expression
	\begin{equation}
		\label{eq:RKBt}
		K_P(x)\ind{B^c}(x/t)=R_P(b^t)(x).
	\end{equation}
	Let $M^t_k$ be the convolution operator  
	\begin{equation*}
		%\label{eq:Mtdef}
		M^t_k f(x)=b^t*f(x).
	\end{equation*}
	It follows from \cite[Section 4]{mopv} that $M_k^t$ is bounded on $L^p$ spaces whenever $1<p<\infty.$ 
	Moreover, in view of \eqref{eq:RKBt} we see that
	\begin{equation*}
		R_P^t f= R_P(b^t)*f=b^t * R_P(f)=M_k^t (R_P f).
	\end{equation*}
	
	Finally, for $f\in \mS,$ $P\in \mH_k,$  and $x\in \R^d$ the mapping $t\mapsto R_P^t f(x)$ is  continuous on $(0,\infty).$ Thus, also $M_k^t (R_P f)(x)$ is a continuous function of $t>0$ for a.e.\ $x.$ This completes the proof of the proposition.
\end{proof}

As a corollary of  \thref{pro:fact} we see that in order to justify \thref{thm1,thm2} it suffices to control vector and scalar valued maximal functions corresponding to the operators $M^t_k.$ In what follows, for $f\in \mD(k)$  we set
\begin{equation*}
	%\label{eq:Mtmax}
	M^*f(x)=\sup_{t>0}|M^t_k f(x)|.
\end{equation*}
Note that by \thref{pro:fact} 
for $f\in \mD(k)$ we have
\begin{equation}
	\label{eq:MtmaxQ}
	M^*f(x)=\sup_{t\in \Q_+ }|M^t_k f(x)|,
\end{equation}
where $\Q_+$ denotes the set of non-negative rational numbers; hence 
the maximal function $M^*f(x)$ is measurable, although possibly being infinite for some $x.$
\begin{theorem}
	\thlabel{thm1'}
Let $k\in \No.$ For each $p\in (1,\infty)$  there is a constant $A(p,k)$ independent of the dimension $d$ and such that for any $S\in\N$  we have
	\begin{equation*} 
		%\label{eq:thm1'}
		\norm{\left(\sum_{s=1}^S |M^* f_s|^2\right)^{1/2}}_{p} \leqslant A(p,k) \norm{ \left(\sum_{s=1}^S | f_s|^2\right)^{1/2}}_{p},
	\end{equation*}
	whenever $f_1,\ldots,f_S \in \mD(k).$ 
	Moreover, for fixed $k$ we have $A(p,k)=O(p^{3+k/2})$ as $p\to \infty$ and $A(p,k)=O((p-1)^{-3-k/2})$ as $p\to 1.$
\end{theorem}
\begin{theorem} \thlabel{thm2'}
	Let $k\in \No.$ For each $p\in (1,\infty)$  there is a constant $ B(p,k)$ independent of the dimension $d$ and such that
	\begin{equation*} 
		%\label{eq:thm2'}
		\norm{ M^* f}_{p} \leqslant B(p,k) \norm{  f}_{p},
	\end{equation*}
	whenever $f\in \mD(k).$  Moreover, for fixed $k$ we have $B(p,k)\lesssim (p^*)^{2+k/2}. $  
\end{theorem}

\thref{thm1',thm2'} together with \thref{pro:fact}  imply \thref{thm1,thm2} with the same values of constants $A(p,k)$ and $B(p,k).$ This is done first for $f\in \mS,$ and then by density for all $f\in L^p.$ Therefore, from now on we focus on proving \thref{thm1',thm2'}.

\section{Averaging}
\label{sec:av}

In this section we describe the averaging procedure. This will allow us to pass from $M^*$ to another maximal function that is better suited for an application of the method of rotations in Section \ref{sec:mr}. Before moving on, we establish some notation. For a multi-index $j = (j_1, \dots, j_k) \in \{1, \dots, d\}^k$ by $R_j$ we denote the Riesz transform associated with the monomial $P_k(x) = x_{j_1} \cdots x_{j_k}$; the truncated transform $R_j^t$ and the maximal transform $R_j^*$ are defined analogously. We will also abbreviate
\[
x_j = x_{j_1} \cdots x_{j_k} \quad \text{and} \quad x_j^n = x_{j_1}^n \cdots x_{j_k}^n.
\]
As we will be mainly interested in multi-indices with different components, we define $I = \{j \in \{1, \dots, d\}^k: j_k \neq j_l \text{ for } k \neq l \}$.

The averaging procedure will provide an expression  for $M^t$ in terms of the Riesz transforms $R$ and $R^t$ postulated in \eqref{eq:lemAp}.
For $f\in L^p,$ $1<p<\infty,$ denote
\begin{equation}
	\label{eq:Rt}
	R^tf:=\sum_{j \in I} R_j^t R_jf\qquad\textrm{and let}\qquad 	R^* f:=\sup_{t\in \Q_+}\left|R^t f\right|.
\end{equation}
Note that both $R^t$ and $R^*$ are well defined on all $L^p$ spaces. Indeed, $R_j^t$ and $R_j$ are $L^p$ bounded and the supremum in the definition of $R^*$ runs over a countable set thus defining a measurable function.

Let $SO(d)$ be the special orthogonal group in dimension $d.$ Since it is compact, it has a bi-invariant Haar measure $\mu$ such that $\mu(SO(d))=1.$ For $U\in SO(d)$ and a sublinear operator $T$ on $L^2$ we denote by $T_U$ the conjugation by $U,$ i.e. the operator acting via 	\begin{equation}\label{conj_U_def}
T_Uf(x)=T(f (U^{-1}\cdot))(Ux).
	\end{equation}

\begin{pro} \thlabel{pro:av}
	Let $k\in \No.$ Then there is a constant $C(d,k)$ such that 
	\begin{equation} \label{eq:lemAplem}
		M^t_k f(x)=C(d,k)\int_{SO(d)} [(R^t)_U f](x) d\mu(U)
	\end{equation}
	for all $t>0$ and $f\in \mD(k).$ Moreover, $C(d,k)$ has an estimate from above by a constant that depends only on $k$ but not on the dimension $d,$ so that 
	\begin{equation} \label{eq:lemA}
		\left(\sum_{s=1}^S |M^* f_s(x)|^2\right)^{1/2} \lesssim \int_{SO(d)}  \left(\sum_{s=1}^S \left|[(R^*)_U f_s](x)\right|^2\right)^{1/2} \, d\mu(U),
	\end{equation}
	for $S\in \N$ and $f_1,\ldots,f_S\in \mD(k).$ 
\end{pro}
\begin{proof}
	Let  $A$ be the operator
	\begin{equation}
		\label{eq:Adef}
		A = \sum_{j \in I} R_j^2 ,
	\end{equation}
	which by \eqref{eq:m} means that its multiplier symbol equals
	\[
	a(\xi) =(-i)^{2k} \sum_{j \in I} \frac{\xi_j^2}{\abs{\xi}^{2k}}=-\sum_{j \in I} \frac{\xi_j^2}{\abs{\xi}^{2k}}.
	\]
	Let  $\widetilde{A}$ be the operator with the multiplier symbol
	\begin{equation}
		\label{eq:mtil}
		\widetilde{a}(\xi) := \int_{SO(d)} a(U\xi) d\mu(U)=- \sum_{j \in I} \int_{SO(d)} \frac{\left( (U\xi \right)_j)^2}{\abs{\xi}^{2k}} d\mu(U).
	\end{equation}
	Then
	$\widetilde{a}$ being
	radial and homogeneous of order $0$ is constant. 
	
	The first step in the proof of the proposition is to show that
	\begin{equation}
		\label{eq:lemA1}
		|\widetilde{a}|\sim 1
	\end{equation}
	uniformly in the dimension $d.$ Note that each of the integrals on the right hand side of \eqref{eq:mtil} has the same value independently of $j\in I,$
	so that
	\[
	\widetilde{a}(\xi)= -\abs{I} \int_{SO(d)} \frac{(\left(U\xi \right)_{(1,\ldots,k)})^2}{\abs{\xi}^{2k}} d\mu(U);
	\]
	here $|I|$ stands for the number of elements in $I.$ Since $\tilde{a}$ is radial, integrating the above expression over the unit sphere $S^{d-1}$ with respect to the normalized surface measure $d\omega$ we obtain
	\begin{equation} \label{eq0}
		\widetilde{a} = -\abs{I} \int_{S^{d-1}} \omega_1^2 \cdots \omega_k^2 \ d\omega.
	\end{equation}
 Since $k$ is fixed, by an elementary argument we get $|I|=d!/(d-k)!\sim d^k$. Thus it remains to show that
		\begin{equation}
		\label{eq:lemA1'}
		\int_{S^{d-1}} \omega_1^2 \cdots \omega_k^2 \ d\omega \sim d^{-k}
	\end{equation}

	Formula \eqref{eq:lemA1'} is given in \cite[(10)]{sykora}. It can be also easily
	computed by the method from \cite[Chapter 3.4]{Ho}; for the sake of completeness we provide a brief argument.
	Consider the integral $J=\int_{\R^d}x_1^2...x_k^2e^{-|x|^2}dx$. Since $J$ is a product of the one-dimensional integrals we calculate $J=\Gamma \left(\frac{3}{2} \right)^k \Gamma \left(\frac{1}{2}\right)^{d-k},$ while using polar coordinates
	gives $J=S_{d-1}\int_{S^{d-1}} \omega_1^2 \cdots \omega_k^2 \ d\omega\int_0^\infty r^{2k+d-1}e^{-r^2}dr$, where $S_{d-1}$ is defined by \eqref{eq:Sd-1}. 
	Altogether we have justified that 
	\[
	\int_{S^{d-1}} \omega_1^2 \cdots \omega_k^2 \ d\omega\sim  \frac{ \Gamma \left(\frac{1}{2} \right)^{d-k}}{S_{d-1}\Gamma\left( k+\frac{d}{2} \right)}.
	\]
	Since $k$ is fixed and $d$ is arbitrarily large, using \eqref{eq:Sd-1}, Stirling's formula for the $\Gamma$ function \eqref{StirF}
	and the known identity  $\Gamma(1/2)=\sqrt{\pi}$  we obtain 
	\begin{align*}
		\int_{S^{d-1}} \omega_1^2 \cdots \omega_k^2 \ d\omega &\sim \frac{ \sqrt{k+\frac{d}{2}} \left( \frac{d}{2e} \right)^{d/2} }{\sqrt{\frac{d}{2}} \left( \frac{k+\frac{d}{2}}{e} \right)^{k+d/2}} \\
		&\sim \frac{ e^{-d/2}}{e^{-k-d/2}} \left( \frac{k+\frac{d}{2}}{d/2} \right)^{-d/2} \left( k+\frac{d}{2} \right)^{-k} \\
		&\sim   d^{-k}
	\end{align*}
	This gives \eqref{eq:lemA1'} and concludes the proof of \eqref{eq:lemA1}.

	Let now $m^t$ be the multiplier symbol of $M^t_k.$ Then, from \thref{pro:fact}  we see that $m^t=\hat{b^t}$ is radial, so that
	\begin{align*}
		m^t(\xi)&=\tilde{a}^{-1} \tilde{a}\, m^t(\xi)=\tilde{a}^{-1} \int_{SO(d)} m^t(\xi)a(U\xi) d\mu(U)\\
		&=\tilde{a}^{-1} \int_{SO(d)} m^t(U\xi)a(U\xi) d\mu(U).
	\end{align*} 
	Using properties of the Fourier transform the above equality implies that
	\begin{align*}
		M^t f(x)=\tilde{a}^{-1}\int_{SO(d)}\, [(M^t A)_U](f)(x)\,d\mu(U).
	\end{align*}
	Recalling \eqref{eq:Adef} we apply \eqref{eq:fact} from \thref{pro:fact}   and obtain
	\[
	M^t A=\sum_{j\in I} M^t R_j R_j=\sum_{j\in I}R_j^t R_j=R^t;
	\]
	here an application of \eqref{eq:fact}  is allowed since each $R_j$ corresponds to the monomial $x_j$  which is in $\mH_k$ when $j\in I.$   In summary, we justified that
	\begin{equation}
		\label{eq:Mtexpp}
		M^t f(x)=\tilde{a}^{-1}\int_{SO(d)}\, [(R^t)_U](f)(x)\,d\mu(U),\qquad f\in\mD(k),
	\end{equation}
	which is \eqref{eq:lemAplem} with $C(d,k)=\tilde{a}^{-1}.$ 
	
	It remains to justify \eqref{eq:lemA}. This follows from \eqref{eq:MtmaxQ}, \eqref{eq:Mtexpp}, and \eqref{eq:lemA1}, together with the norm inequality
	\[
	\norm{\int_{SO(d)}\, F_{s,t}(U)\,d\mu(U)}_X\leqslant \int_{SO(d)}\,\norm{ F_{s,t}(U)}_X\,d\mu(U);
	\]
	on the Banach space $X=\ell^2(\{1,\ldots,S\};\ell^{\infty}(\Q_+)),$ with $F_{s,t}(U)=(R^t)_U(f_s)(x)$ and $x$ being fixed. 
	
	The proof of \thref{pro:av} is thus completed.
\end{proof}

Since conjugation by $U\in SO(d)$  is an isometry on all $L^p$ spaces, $1\leqslant p\leqslant \infty,$ we have, for $f_s\in \mD(k)$ 
\[
\norm{\left(\sum_{s=1}^S \left|[(R^*)_U f_s]\right|^2\right)^{1/2} }_p=\norm{\left(\sum_{s=1}^S [R^* f_s]^2\right)^{1/2} }_p.
\]

Thus, in view of $\mu(SO(d))=1$ and Minkowski's integral inequality   \thref{pro:av} eq.\ \eqref{eq:lemA} allows us to deduce \thref{thm1',thm2'} from the two theorems below. Note that, by our convention, the implicit constants from \thref{thm1'',thm2''} transfer to $A(p,k)\leqslant C_{}(k)(p^*)^{3+k/2}$ and $B(p,k)\leqslant C_{}(k)(p^*)^{2+k/2},$ in  \thref{thm1',thm2'} hence, also in \thref{thm1,thm2}; here $C(k)$ denotes a constant that depends only on $k.$

\begin{theorem}
	\thlabel{thm1''}
	Let $k\in \No$ 	and take $p \in (1, \infty).$ Then, for $f_1,\ldots,f_S \in L^p$ it holds 
	\begin{equation*} 
		%\label{eq:thm1''}
		\norm{\left(\sum_{s=1}^S |R^* f_s|^2\right)^{1/2}}_{p} \lesssim (p^*)^{3+k/2} \norm{ \left(\sum_{s=1}^S | f_s|^2\right)^{1/2}}_{p}.
	\end{equation*}
\end{theorem}
\begin{theorem} \thlabel{thm2''}
		Let $k\in \No$ 	and take $p \in (1, \infty).$  Then for  $f\in L^p$ it holds 
	\begin{equation*} 
		%\label{eq:thm2''}
		\norm{ R^* f}_{p} \lesssim (p^*)^{2+k/2}  \norm{  f}_{p}.
	\end{equation*}
\end{theorem}

   \section{The method of rotations --- bounds for $R^*$}
   \label{sec:mr}
   
The goal of this section is to prove \thref{thm1'',thm2''}. This will be done by the method of rotations together with a number of duality arguments. In proving \thref{thm1''} we shall also need Khintchine's inequality. 

Before going further we need a lemma on the explicit $L^p$ bounds for the square function corresponding to the Riesz transforms $R_j,$ $\vj \in I.$ This will be derived from \cite[Th\'eor\`eme 2]{duo_rubio}. The key observation in the proof of \thref{lem:Rjduru} is that $|I|\sim \dim \mH_k$ (more precisely, $|I|\sim k! \dim \mH_k$). We provide details for the convenience of the reader. A version of \thref{lem:Rjduru} can be also deduced from \eqref{eq:Riesz0} together with an iterative applications of Khintchine's inequality. However, such an approach produces  worse constants than  \cite[Th\'eor\`eme 2]{duo_rubio}.
\begin{lemma}
	\thlabel{lem:Rjduru}
	Let $1<p<\infty.$ Then, for $f\in L^p$ we have
	\begin{equation*}
		%\label{eq:Rjduru}
		 \norm{\left( \sum_{\vj \in I} \left( R_jf \right)^2 \right)^{1/2}}_p \lesssim \max(p,(p-1)^{-1-k/2})\|f\|_p.
	\end{equation*}    
\end{lemma}
\begin{proof}
	By \eqref{eq:lemA1} and \eqref{eq0} we see that 
	$\int_{S^{d-1}} (x_j)^2\sim \frac{1}{|I|}.$ Additionally, since 
	\begin{align*}
	\dim \mH_k&={d+k -1 \choose k}-{d+k -3 \choose k-2}\\
	&=\frac{(d+k-3)!}{(k-2)!(d-1)!}\left(\frac{(d+k-2)(d+k-1)}{(k-1)k}-1\right)\end{align*}
	we see that $\dim \mH_k=a(d,k)\sim d^{k},$ with an implicit constant depending on $k$ but not on the dimension $d.$ Since $|I|=d!/(d-k)!$ we thus have $|I|\sim d^k\sim a(d,k),$ so that
	\begin{equation*}
		%\label{eq:insumm}
		\|x_j\|_{L^2(S^{d-1},d\omega)}\sim \frac{1}{\sqrt{a(d,k)}}.
	\end{equation*}

Defining
\[
Y_j(x)=\frac{1}{\sqrt{a(d,k)}\|x_j\|_{L^2}} x_j,\qquad \vj \in I,
\]
we thus see that
\begin{equation}
	\label{eq:wereach}
Y_j(x)=c(d,k)x_j,\qquad R_{Y_j}=c(d,k)R_j,
\end{equation}
where $c(d,k)\sim 1.$ Moreover, $Y_j,$ $\vj \in I$, are pairwise orthogonal and satisfy 
\[
\int_{S^{d-1}}|Y_j(\omega)|^2\,d\omega=\frac{1}{a(d,k)}.
\] 
Completing the set $\{Y_j\}_{\vj \in I}$ to an orthonormal basis of $\mH_k$ we obtain a larger set $\{Y_j\}_{j\in J},$ where $I\subseteq J$ and $|J|=a(d,k).$    
 Therefore, from  \cite[Th\'eor\`eme 2]{duo_rubio} we reach
 \begin{equation*}
 	\norm{\left( \sum_{j \in J} \left( R_{Y_j}f \right)^2 \right)^{1/2}}_p \lesssim \max(p,(p-1)^{-1-k/2})\|f\|_p,
 \end{equation*}  
and an application of \eqref{eq:wereach} completes the proof of the lemma.
\end{proof}

Having justified \thref{lem:Rjduru} we move on to the proof of \thref{thm2''}.

\begin{proof}[Proof of \thref{thm2''}]
	From Lebesgue's monotone convergence theorem it follows that we may restrict the supremum in the definition \eqref{eq:Rt} of $R^*$ to a finite set, say $\{t_1,\ldots,t_N\},$ as long as our final estimate is independent of $N.$ 
	
	  Let $q$ be the conjugate exponent to $p,$ i.e.\ such that $1/p+1/q=1.$ Using duality between the spaces $L^p(\R^d;\ell^\infty(\{t_1,\ldots,t_N\}))$ and $L^q(\R^d;\ell^1({t_1,\ldots,t_N}))$ our task is reduced to the following equivalent inequality 
	\begin{equation*} 
		%\label{eq4}
		\abs{\int_{\R^d} \sum_{n=1}^N \sum_{\vj \in I} R^{t_n}_j R_j f(x) g_n(x) dx} \lesssim \max(p^{2+k/2},(p-1)^{-2-k/2}) \norm{f}_p \norm{\sum_{n=1}^N \abs{g_n}}_q, 
	\end{equation*}
where $\{g_n\}\in L^q(\R^d;\ell^1({t_1,\ldots,t_N})).$ 
	Since for each $\vj \in I$ the monomial $P_j(x)=x_j$ is a real-valued function that satisfies $P_j(-x)=-P_j(x)$, by \eqref{eq:R} the operators $R_j^t,$ $\vj \in I$ are skew-adjoint, i.e.\ $(R_j^t)^*=-R_j^t,$ $\vj \in I$.  Hence, by Cauchy--Schwarz inequality, H\"{o}lder's inequality, and \thref{lem:Rjduru} we get
	\begin{align*} 
		%\label{eq5}
		&\abs{\int_{\R^d} \sum_{n=1}^N \sum_{\vj \in I} R_j^{t_n} R_j f(x) g_n(x) dx} = \abs{\int_{\R^d} \sum_{\vj \in I} R_j f(x) \cdot \sum_{n=1}^N R_j^{t_n}g_n(x) dx} \nonumber \\
		&\leqslant \int_{\R^d} \left( \sum_{\vj \in I} \left( R_jf(x) \right)^2 \right)^{1/2} \left( \sum_{\vj \in I} \left( \sum_{n=1}^N R_j^{t_n} g_n(x) \right)^2 \right)^{1/2} dx \nonumber \\
		&\leqslant \norm{\left( \sum_{\vj \in I} \left( R_jf \right)^2 \right)^{1/2}}_p \norm{\left( \sum_{\vj \in I} \left( \sum_{n=1}^N R_j^{t_n} g_n \right)^2 \right)^{1/2}}_q\\
		&\lesssim \max(p,(p-1)^{-1-k/2})\|f\|_p\cdot \norm{\left( \sum_{\vj \in I} \left( \sum_{n=1}^N R_j^{t_n} g_n \right)^2 \right)^{1/2}}_q.
	\end{align*}

Therefore we can focus on proving that
\begin{equation}
	\label{eq:st}
	\norm{\left( \sum_{\vj \in I} \left( \sum_{n=1}^N R_j^{t_n} g_n \right)^2 \right)^{1/2}}_q\lesssim \max(p^{1+k/2},(p-1)^{-1})\cdot \norm{\sum_{n=1}^N \abs{g_n}}_q
\end{equation}
Here we use the method of rotations, specifically \cite[5.2.20]{grafakos}, to the truncated Riesz transforms $R_j^t$ obtaining
	\begin{equation} 
		\label{eq:rot-1}
	R_j^t f(x) = \gamma_k' \int_{S^{d-1}} \theta_j H_\theta^t f(x) d\theta.
\end{equation}
Here $\gamma_k' = \frac{\pi}{2}\gamma_k$, $d\theta$ is the unnormalized surface measure on $S^{d-1}$ (i.e. $\int_{S^{d-1}} d\theta = S_{d-1} = \frac{2\pi^{d/2}}{\Gamma\left( \frac{d}{2} \right)}$), and $H_\theta^t$ is the truncated directional Hilbert transform (see \cite[Section 5.2.3]{grafakos} for more details). Recall that  $\theta_j=\theta_{j_1}\cdots \theta_{j_k}.$ In terms of the normalized surface measure $d\omega$ on $S^{d-1}$ equality \eqref{eq:rot-1} becomes
	\begin{equation} \label{eq:rot}
		R_j^t f(x) = \frac{\pi \Gamma((k+d)/2)}{\Gamma(k/2)\Gamma(d/2)}\int_{S^{d-1}} \omega_j H_\omega^t f(x) d\omega.
	\end{equation}
Note that since $k$ is fixed and $d$ is large, in view of \eqref{StirF} we have
\begin{equation}
	\label{eq:approxGkd}
	\frac{\pi \Gamma((k+d)/2)}{\Gamma(k/2)\Gamma(d/2)}\sim d^{k/2}.
\end{equation}
 Now, take numbers $\lambda_j(x),$ $j\in I$, such that
	\begin{equation} \label{eq7}
		\left( \sum_{\vj \in I} \left( \sum_{n=1}^N R_j^{t_n} g_n(x) \right)^2 \right)^{1/2} = \sum_{\vj \in I} \lambda_j(x) \sum_{n=1}^N R_j^{t_n} g_n(x), \qquad \sum_{\vj \in I} \lambda_j^2 = 1.
	\end{equation}
	
	Using \eqref{eq7}, \eqref{eq:rot}, and \eqref{eq:approxGkd} followed by H\"older's inequality we obtain
	\begin{align} \label{eq8}
		&\norm{\left( \sum_{\vj \in I} \left( \sum_{n=1}^N R_j^{t_n} g_n \right)^2 \right)^{1/2}}_q^q = \int_{\R^d} \abs{\sum_{\vj \in I} \lambda_j(x) \sum_{n=1}^N R_j^{t_n} g_n(x)}^q dx \nonumber \\
		&\lesssim^q  d^{kq/2}\int_{\R^d} \abs{\int_{S^{d-1}} \sum_{\vj \in I} \lambda_j(x) \omega_j \sum_{n=1}^N H_\omega^{t_n} g_n(x) d\omega }^q dx \nonumber \\
		&\leqslant d^{kq/2} \int_{\R^d} \left(\int_{S^{d-1}} \abs{\sum_{\vj \in I} \lambda_j(x) \omega_j}^p d\omega \right)^{q/p} \int_{S^{d-1}} \abs{\sum_{n=1}^N H_\omega^{t_n} g_n(x)}^q d\omega dx;
	\end{align}
with  the meaning of $\lesssim^q$ being explained in Section \ref{sec:not} item (7).

We deal with the first inner integral in \eqref{eq8}. Using \cite[Lemme, p. 195]{duo_rubio}, the formula $\sum_{j\in I}\lambda_j(x)^2=1,$ and \eqref{eq:lemA1'} we get
	\begin{equation}
		\label{eq:calcu}
		\begin{split}
		&\left(\int_{S^{d-1}} \abs{\sum_{\vj \in I} \lambda_j(x) \omega_j}^p d\omega \right)^{1/p} \lesssim p^{k/2} \left(\int_{S^{d-1}} \abs{\sum_{\vj \in I} \lambda_j(x) \omega_j}^2 d\omega \right)^{1/2} \\
		&= p^{k/2}  \left(\int_{S^{d-1}} \sum_{\vj \in I} \lambda_j(x)^2 \omega_j^2 d\omega \right)^{1/2} = p^{k/2} \left(\sum_{\vj \in I} \lambda_j(x)^2 \int_{S^{d-1}} \omega_{(1,\dots,k)}^2 d\omega \right)^{1/2} \\
		&= p^{k/2} \left(\int_{S^{d-1}} \omega_1^2 \cdots \omega_k^2 \ d\omega \right)^{1/2} \sim p^{k/2} d^{-k/2},
		\end{split}
	\end{equation}
where the first equality above follows from the observation that if we expand the squared sum, then only the diagonal terms contribute non-zero integrals over $S^{d-1}$. Note that an application of \cite[Lemme, p. 195]{duo_rubio} is permitted here, since for each fixed $x\in \R^d$ the function $\sum_{\vj \in I} \lambda_j(x) \omega_j$ belongs to $\mH_k.$ At this point it is again important that $j\in I.$ Inequality \eqref{eq:calcu} implies
\begin{equation}
	\label{eq:calcu'}
\left(\int_{S^{d-1}} \abs{\sum_{\vj \in I} \lambda_j(x) \omega_j}^p d\omega \right)^{q/p}\lesssim^q p^{kq/2}d^{-kq/2}.
\end{equation}

	Collecting \eqref{eq8} and \eqref{eq:calcu'}, 
	%and using {\color{green} H\"older's inequality (w którym miejscu?)} 
	we see that the proof of  \eqref{eq:st} will be finished if we show that
	\begin{equation*} 
		%\label{eq10}
		\norm{\sum_{n=1}^N H_\omega^{t_n} g_n}_q \lesssim p^* \norm{\sum_{n=1}^N \abs{g_n}}_q,
	\end{equation*}
uniformly in $\omega \in S^{d-1}.$
	To this end we use duality between the spaces $L^q$ and $L^p$ which lets us write the following equivalent inequality
	\begin{equation} \label{eq11}
		\abs{\int_{\R^d} \sum_{n=1}^N H_\omega^{t_n} g_n(x) h(x) dx} \lesssim p^* \norm{\sum_{n=1}^N \abs{g_n}}_q \norm{h}_p, \quad h \in L^p.
	\end{equation}
	Using H\"{o}lder's inequality on the left-hand side of \eqref{eq11} we arrive at
	\begin{align*}
		&\abs{\int_{\R^d} \sum_{n=1}^N H_\omega^{t_n} g_n(x) h(x) dx} \leqslant \int_{\R^d} \sum_{n=1}^N \abs{g_n(x) H_\omega^{t_n} h(x)} dx \\
		&\leqslant \int_{\R^d} \max_{1 \leqslant n \leqslant N} \abs{H_\omega^{t_n} h(x)} \sum_{n=1}^N \abs{g_n(x) } dx \leqslant \norm{H^*_\omega h}_p \norm{\sum_{n=1}^N \abs{g_n}}_q \\
		&\leqslant \norm{H^*}_{p} \norm{h}_p \norm{\sum_{n=1}^N \abs{g_n}}_q,
	\end{align*}
	where $H^*_\omega$ is the maximal directional Hilbert transform on $\R^d$ and $H^*$ is the maximal Hilbert transform on $\R$.  In the last inequality we used the fact that for $\omega\in S^{d-1}$ it holds $\norm{H^*_\omega}_{L^p(\R^d)} = \norm{H^*}_{L^p(\R)}$. Finally $H^*$ is bounded on $L^p$ and $\|H^*\|_p\lesssim p^*,$ see  \cite[Theorem 4.2.4, eq.\ (4.2.4)]{grafakos_modern}. This completes the proof of \eqref{eq11} and hence also the proof of  \thref{thm2''}.

\end{proof}

We shall now prove \thref{thm1''}. The main idea is similar to the one used in the proof of \thref{thm2''}. The computations, however, are more involved, mainly because of a need for extra (Khintchine's) inequalities.
\begin{proof}[Proof of \thref{thm1''}]
As in the proof of \thref{thm2''} we reduce the supremum to a finite sequence of positive numbers $t_1, \dots, t_N$, which leaves us with the goal to prove
	\begin{equation} \label{eq22}
		\norm{\left(\sum_{s=1}^S \sup_{1 \leqslant n \leqslant N} \abs{ \sum_{\vj \in I} R_j^{t_n} R_j f_s}^2 \right)^{1/2}}_p \lesssim  \max(p^{3+k/2},(p-1)^{-3-k/2}) \norm{\left( \sum_{s=1}^S f_s^2 \right)^{1/2}}_p.
	\end{equation}
	
	We use duality between the spaces $L^p(\R^d;E_{\infty})$ and $L^q(\R^d;E_1)$, with  $$E_{\infty}=\ell^2(\{1,\ldots,S\};\ell^{\infty}(\{t_1,\ldots,t_N\})),\quad E_1=\ell^2(\{1,\ldots,S\};\ell^{1}(\{t_1,\ldots,t_N\})),$$ and $p$ and $q$ being conjugate exponents. This allows us to write \eqref{eq22} in the following equivalent form
	\begin{equation} \label{eq23}
		\begin{aligned}
			&\abs{\int_{\R^d} \sum_{n=1}^N \sum_{s=1}^S \sum_{\vj \in I} R_j^{t_n} R_j f_s(x) g_{s,n}(x) dx} \\
			&\lesssim \max(p^{3+k/2},(p-1)^{-3-k/2}) \norm{\left( \sum_{s=1}^S f_s^2 \right)^{1/2}}_p \norm{\left( \sum_{s=1}^S \left(\sum_{n=1}^N \abs{g_{s,n}}\right)^2 \right)^{1/2}}_q
		\end{aligned}
	\end{equation}
	for any $g_{s,n} \in L^q(\R^d;E_1)$. Since $(R_j^{t_n})^*=-R_j^{t_n}$ for $j\in I,$  an application of Cauchy--Schwarz inequality and H\"{o}lder's inequality gives
	\begin{align*} 
		%\label{eq24}
		&\abs{\int_{\R^d} \sum_{s=1}^S \sum_{n=1}^N \sum_{\vj \in I} R_j^{t_n} R_j f_s(x) g_{s,n}(x) dx} = \abs{\int_{\R^d} \sum_{\vj \in I} \sum_{s=1}^S R_j f_s(x) \cdot \sum_{n=1}^N R_j^{t_n}g_{s,n}(x) dx} \nonumber \\
		&\leqslant \int_{\R^d} \left( \sum_{\vj \in I} \sum_{s=1}^S \left( R_j f_s(x) \right)^2 \right)^{1/2} \left( \sum_{\vj \in I} \sum_{s=1}^S \left( \sum_{n=1}^N R_j^{t_n} g_{s,n}(x) \right)^2 \right)^{1/2} dx \nonumber \\
		&\leqslant \norm{\left( \sum_{\vj \in I} \sum_{s=1}^S \left( R_j f_s \right)^2 \right)^{1/2}}_p \norm{\left( \sum_{\vj \in I} \sum_{s=1}^S \left( \sum_{n=1}^N R_j^{t_n} g_{s,n} \right)^2 \right)^{1/2}}_q.
	\end{align*}
	Hence in order to prove \eqref{eq23} it is enough to show that
	\begin{equation} \label{eq25}
		\norm{\left( \sum_{\vj \in I} \sum_{s=1}^S \left( R_j f_s \right)^2 \right)^{1/2}}_p \lesssim\max(p^{\frac32},(p-1)^{-\frac32-k/2}) \norm{\left( \sum_{s=1}^S f_s^2 \right)^{1/2}}_p
	\end{equation}
	and
	\begin{equation} \label{eq26}
		\begin{split}
		&\norm{\left( \sum_{\vj \in I} \sum_{s=1}^S \left( \sum_{n=1}^N R_j^{t_n} g_{s,n} \right)^2 \right)^{1/2}}_q \\
		&\lesssim \max(p^{\frac32+k/2},(p-1)^{-\frac32}) \norm{\left( \sum_{s=1}^S \left(\sum_{n=1}^N \abs{g_{s,n}}\right)^2 \right)^{1/2}}_q,
		\end{split}
	\end{equation}
uniformly in $t_1,\ldots,t_N.$

It turns out that \eqref{eq26} implies \eqref{eq25}. Indeed, switching the roles of $p$ and $q$ and  taking $g_{s,1} = f_s$ and all other $g_{s,n} = 0$ in \eqref{eq26} we obtain a variant of \eqref{eq25} with $R_j$ replaced by $R_j^{t_1},$ namely 
		\begin{equation} \label{eq25'}
		\norm{\left( \sum_{\vj \in I} \sum_{s=1}^S \left( R_j^{t_1} f_s \right)^2 \right)^{1/2}}_p \lesssim \max(p^\frac32,(p-1)^{-\frac32-k/2}) \norm{\left( \sum_{s=1}^S f_s^2 \right)^{1/2}}_p.
	\end{equation}
Now, since $\lim_{t_1\to 0} R_j^{t_1}f_s=R_j f_s$,  an application of Fatou's lemma shows that \eqref{eq25'}, being uniform in $t_1>0,$ implies \eqref{eq25} with the same constants.

Therefore, in what follows we will focus on establishing \eqref{eq26}.
	Similarly to the proof of \thref{thm2''} we take numbers $\lambda_{s,j}(x),$ $s\in \{1,\ldots,S\},$ $j\in I,$ such that
	\begin{equation*} 
		%\label{eq27}
		\left( \sum_{\vj \in I} \sum_{s=1}^S \left( \sum_{n=1}^N R_j^{t_n} g_{s,n}(x) \right)^2 \right)^{1/2} = \sum_{\vj \in I} \sum_{s=1}^S \lambda_{s,j}(x) \sum_{n=1}^N R_j^{t_n} g_{s,n}(x)
	\end{equation*}
	and $\sum_{\vj \in I} \sum_{s=1}^S \lambda_{s,j}^2 = 1$ and we use the method of rotations \eqref{eq:rot}. Together with \eqref{eq:approxGkd} this gives
	\begin{align} \label{eq28}
		&\norm{\left( \sum_{\vj \in I} \sum_{s=1}^S \left( \sum_{n=1}^N R_j^{t_n} g_{s,n} \right)^2 \right)^{1/2}}_q^q = \int_{\R^d} \abs{\sum_{\vj \in I} \sum_{s=1}^S  \sum_{n=1}^N \lambda_{s,j}(x) R_j^{t_n} g_{s,n}(x)}^q dx \nonumber \\
		&\lesssim^q d^{kq/2} \int_{\R^d} \abs{\int_{S^{d-1}} \sum_{\vj \in I} \sum_{s=1}^S \sum_{n=1}^N \lambda_{s,j}(x) \omega_j H_\omega^{t_n} g_{s,n}(x) d\omega }^q dx;
	\end{align}
recall here the meaning of $\lesssim^q$ in Section \ref{sec:not} item (7).
	
	At this point we need to use Khintchine's inequality. Let $\{r_s\},$ $s=1,2,\ldots,$ be the Rademacher functions, see \cite[Appendix C]{grafakos}. These form an orthonormal set on $L^2([0,1])$.
%	\begin{equation} \label{eq:rad_ort}
%		\int_0^1 r_j(\xi) r_k(\xi) d\xi =
%		\begin{cases}
%			1 \quad &\text{if } j = k \\
%			0 \quad &\text{if } j \neq k.
%		\end{cases}
%	\end{equation}
Moreover we have Khintchine's inequality (\cite[Appendix C.2]{grafakos})
	\begin{equation} 
		\label{eq:chinczyn}
		\norm{\sum_{j=1}^\infty a_j r_j}_{L^p([0,1])} \lesssim p^{\frac12} \left( \sum_{j=1}^\infty a_j^2 \right)^{1/2},
	\end{equation}
for any real sequence $(a_s)_{s=1}^\infty$ and $1 \leqslant  p < \infty.$ The explicit bounds on constants in  \eqref{eq:chinczyn} follow from explicit values of the optimal constants established by Haagerup \cite{Ha} together with Stirling's formula \eqref{StirF}.   Using the orthonormality of $\{r_s\}$ we rewrite the right-hand side of \eqref{eq28} as
	\begin{equation}
		\label{eq29'}
		\begin{split}
	&	d^{kq/2} \int_{\R^d} \abs{\int_{S^{d-1}} \sum_{\vj \in I} \sum_{s=1}^S \sum_{n=1}^N \lambda_{s,j}(x) \omega_j H_\omega^{t_n} g_{s,n}(x) d\omega }^q dx  \\
		&= d^{kq/2}\int_{\R^d} \left|\int_{S^{d-1}} \int_0^1 \left( \sum_{\vj \in I} \sum_{s=1}^S \lambda_{s,j}(x) \omega_j r_s(\xi) \right)\right.\\
		 &\left.\hspace{3.4cm}\times \left(\sum_{s=1}^S \sum_{n=1}^N H_\omega^{t_n} g_{s,n}(x) r_s(\xi) \right) d\xi d\omega \right|^q dx  
		 \end{split}
		\end{equation}
	and estimate it using H\"{o}lder's inequality by
	\begin{equation}
		 \label{eq29}
	\begin{split}
		&d^{kq/2}\int_{\R^d} \left( \int_{S^{d-1}} \int_0^1 \abs{\sum_{\vj \in I} \sum_{s=1}^S \lambda_{s,j}(x) \omega_j r_s(\xi)}^p d\xi d\omega \right)^{q/p} \\
		&\hspace{3em}\times \int_{S^{d-1}} \int_0^1 \abs{\sum_{s=1}^S \sum_{n=1}^N H_\omega^{t_n} g_{s,n}(x) r_s(\xi)}^q d\xi d\omega \ dx.
		\end{split}
	\end{equation}
	We shall now estimate the inner integral in the first line of \eqref{eq29}. Here the proof splits into two cases. 
	
	If $p \geqslant 2$, we apply Khintchine's inequality \eqref{eq:chinczyn}, Minkowski's inequality and \cite[Lemme, p.\ 195]{duo_rubio}, 
	%and {\color{green} H\"older's inequality (tu też nie widzę H\"oldera)} 
	obtaining
	\begin{equation*} 
		%\label{eq301}
		\begin{aligned}
			\int_{S^{d-1}} &\int_0^1 \abs{\sum_{\vj \in I} \sum_{s=1}^S \lambda_{s,j}(x) \omega_j r_s(\xi)}^p d\xi d\omega \lesssim^p p^\frac{p}{2} \int_{S^{d-1}} \left( \sum_{s=1}^S \left( \sum_{\vj \in I} \lambda_{s,j}(x) \omega_j \right)^2 \right)^{p/2} d\omega \\
			&\leqslant p^\frac{p}{2} \left( \sum_{s=1}^S \left( \int_{S^{d-1}} \left| \sum_{\vj \in I} \lambda_{s,j}(x) \omega_j \right|^p d\omega \right)^{2/p} \right)^{p/2} \\
			&\lesssim^p p^\frac{p}{2} p^{kp/2} \left(  \sum_{s=1}^S \int_{S^{d-1}} \left( \sum_{\vj \in I} \lambda_{s,j}(x) \omega_j \right)^2 d\omega \right)^{p/2}.
		\end{aligned}
	\end{equation*}
Here an application of \cite[Lemme, p.\ 195]{duo_rubio} is justified since $\omega_j\in \mH_k$ for  $j\in I$ and thus also the sum  $\sum_{\vj \in I} \sum_{s=1}^S \lambda_{s,j}(x) \omega_j r_s(\xi)$ belongs to $\mH_k$ for each fixed $x\in \R^d$ and $\xi \in[0,1].$ Now, using the orthogonality of $\omega_j,$ $j\in I$ we see that
\begin{equation}
	\label{eq:intst}
	\begin{split}
		&\int_{S^{d-1}} \int_0^1 \abs{\sum_{\vj \in I} \sum_{s=1}^S \lambda_{s,j}(x) \omega_j r_s(\xi)}^p d\xi d\omega	\\
		&\lesssim^p p^\frac{p}{2} p^{kp/2}\left( \sum_{s=1}^S \int_{S^{d-1}} \sum_{\vj \in I} \lambda_{s,j}(x)^2 \omega_j^2 \, d\omega \right)^{p/2} = p^ \frac{p}{2}p^{kp/2} \left(\int_{S^{d-1}} \omega_{(1,\dots,k)}^2 \, d\omega \right)^{p/2}.
		\end{split}
\end{equation}    

	If on the other hand $1 < p < 2$, an application of H\"older's inequality together with \eqref{eq:intst} in the case $p=2$ shows that
	\begin{align*}
		&\int_{S^{d-1}} \int_0^1 \abs{\sum_{\vj \in I} \sum_{s=1}^S \lambda_{s,j}(x) \omega_j r_s(\xi)}^p d\xi d\omega \\
		&\leqslant \left( \int_{S^{d-1}} \int_0^1 \abs{\sum_{\vj \in I} \sum_{s=1}^S \lambda_{s,j}(x) \omega_j r_s(\xi)}^2 d\xi d\omega \right)^{p/2} \lesssim \left(\int_{S^{d-1}} \omega_{(1,\dots,k)}^2 \, d\omega \right)^{p/2}.
	\end{align*}

Altogether   \eqref{eq:intst} remains true for all $p\in(1,\infty).$
Thus, using \eqref{eq:lemA1'} we arrive at
	\begin{align*}
		d^{kq/2} &\left( \int_{S^{d-1}} \int_0^1 \abs{\sum_{\vj \in I} \sum_{s=1}^S \lambda_{s,j}(x) \omega_j r_s(\xi)}^p d\xi d\omega \right)^{q/p} \lesssim^q p^\frac{q}{2} p^{kq/2}.
	\end{align*}
Returning to \eqref{eq29'} and \eqref{eq29} we have thus proved
\begin{align*}
&	\left(d^{kq/2} \int_{\R^d} \abs{\int_{S^{d-1}} \sum_{\vj \in I} \sum_{s=1}^S \sum_{n=1}^N \lambda_{s,j}(x) \omega_j H_\omega^{t_n} g_{s,n}(x) d\omega }^q dx\right)^{1/q}\\
&\lesssim p^{\frac12+k/2} \left(\int_{\R^d}\int_{S^{d-1}} \int_0^1 \abs{\sum_{s=1}^S \sum_{n=1}^N H_\omega^{t_n} g_{s,n}(x) r_s(\xi)}^q d\xi d\omega \ dx\right)^{1/q}.
\end{align*}
In view of \eqref{eq28} we now see that \eqref{eq26} will follow if we establish
	\begin{equation} \label{eq31}
		\begin{aligned}
			&\left( \int_{\R^d} \int_{S^{d-1}} \int_0^1 \abs{\sum_{s=1}^S \sum_{n=1}^N H_\omega^{t_n} g_{s,n}(x) r_s(\xi)}^q d\xi d\omega \ dx \right)^{1/q} \\
			&\leqslant \max(p,(p-1)^{-\frac32})\norm{\left( \sum_{s=1}^S \left(\sum_{n=1}^N \abs{g_{s,n}}\right)^2 \right)^{1/2}}_q.
		\end{aligned}
	\end{equation}

	In the reminder of the proof we thus focus on justifying \eqref{eq31}. We use Khintchine's inequality \eqref{eq:chinczyn} on the left-hand side of \eqref{eq31} and get
	\begin{align*}
		&\left(\int_{\R^d} \int_{S^{d-1}} \int_0^1 \abs{\sum_{s=1}^S \sum_{n=1}^N H_\omega^{t_n} g_{s,n}(x) r_s(\xi)}^q d\xi d\omega \ dx \right)^{1/q} \\
		&\lesssim q^\frac12\left(\int_{S^{d-1}}\int_{\R^d}  \left( \sum_{s=1}^S \left( \sum_{n=1}^N H_\omega^{t_n} g_{s,n}(x) \right)^2 \right)^{q/2} dx \, d\omega \right)^{1/q}.
	\end{align*}
	Since $q \lesssim \max(1,(p-1)^{-1})$ our  goal is now to prove the uniform in  $\omega \in S^{d-1}$ estimate
	\begin{equation} \label{eq32}
		\begin{aligned}
			\norm{\left( \sum_{s=1}^S \left( \sum_{n=1}^N H_\omega^{t_n} g_{s,n} \right)^2 \right)^{1/2}}_q \lesssim p^* \, \norm{\left( \sum_{s=1}^S \left(\sum_{n=1}^N \abs{g_{s,n}}\right)^2 \right)^{1/2}}_q.
		\end{aligned}
	\end{equation}
	Using duality between  $L^q(\R^d;\ell^2(\{1,\ldots,S\}))$ and $L^p(\R^d;\ell^2(\{1,\ldots,S\}))$ we write an equivalent inequality
	\begin{equation} \label{eq33}
		\begin{aligned}
			&\abs{\int_{\R^d} \sum_{s=1}^S \sum_{n=1}^N H_\omega^{t_n} g_{s,n}(x) k_s(x) dx}\\
			&\lesssim p^* \norm{\left( \sum_{s=1}^S \left(\sum_{n=1}^N \abs{g_{s,n}} \right)^2 \right)^{1/2} }_q \norm{ \left( \sum_{s=1}^S \abs{k_s}^2 \right)^{1/2}}_p
		\end{aligned}
	\end{equation}
where $k_s \in L^p(\R^d;\ell^2(\{1,\ldots,S\}))$. Since $(H_\omega^{t_n})^*=-H_{\omega}^{t_n}$, Cauchy--Schwarz inequality and H\"{o}lder's inequality give
	\begin{align} \label{eq34}
		&\abs{\int_{\R^d} \sum_{s=1}^S \sum_{n=1}^N H_\omega^{t_n} g_{s,n}(x) k_s(x) dx} = \abs{\int_{\R^d} \sum_{s=1}^S \sum_{n=1}^N g_{s,n}(x) H_\omega^{t_n} k_s(x) dx} \nonumber \\
		&\leqslant \int_{\R^d} \sum_{s=1}^S \sum_{n=1}^N \abs{g_{s,n}(x)} H_\omega^* k_s(x) dx \nonumber \\
		&\leqslant \int_{\R^d} \left(\sum_{s=1}^S \left( \sum_{n=1}^N \abs{g_{s,n}(x)} \right)^2 \right)^{1/2} \left(\sum_{s=1}^S \left( H_\omega^* k_s(x) \right)^2 \right)^{1/2} dx \nonumber \\
		&\leqslant \norm{\left( \sum_{s=1}^S \left(\sum_{n=1}^N \abs{g_{s,n}} \right)^2 \right)^{1/2} }_q \norm{ \left( \sum_{s=1}^S \abs{H_\omega^* k_s}^2 \right)^{1/2}}_p.
	\end{align}
	Comparing \eqref{eq33} and \eqref{eq34} we see that \eqref{eq32} will follow if we justify
	\begin{equation*} 
		%\label{eq35}
		\norm{ \left( \sum_{s=1}^S \abs{H_\omega^* k_s}^2 \right)^{1/2}}_p \lesssim  p^* \norm{ \left( \sum_{s=1}^S \abs{k_s}^2 \right)^{1/2}}_p.
	\end{equation*}
By rotational invariance the above inequality reduces to its one-dimensional case
	\begin{equation}
		\label{eq:Hmvod} 
	\norm{ \left( \sum_{s=1}^S \abs{H^* k_s}^2 \right)^{1/2}}_{L^p(\R)} \lesssim p^*\norm{ \left( \sum_{s=1}^S \abs{k_s}^2 \right)^{1/2}}_{L^p(\R)}.
\end{equation}
Inequality \eqref{eq:Hmvod} can be deduced along the lines of \cite[Section 5.6]{grafakos}. We sketch the argument for the convenience of the reader.

Let $\varphi\colon \R\to \R$ be a smooth even function which satisfies $\varphi(x)=1$ for $|x|<2$ and $\varphi(x)=0$ for $|x|>4.$ Denoting $\varphi_t(x)=\varphi(x/t),$  $\chi_t(x)=\ind{(t,\infty)}(|x|)x^{-1}$  we have the pointwise estimate
\begin{equation}
	\label{eq:H*split}
	\begin{split}
	H^*f(x)&\leqslant \frac{1}{\pi}\sup_{t>0}|(\varphi_t\chi_t *f)(x)|+\frac{1}{\pi}\sup_{t>0}|((1-\varphi_t)\chi_t) *f(x)|\\
	&=:H^*_{\varphi}f(x)+H^*_{1-\varphi}f(x)\\
	&\lesssim \mathcal M f(x)+ H^*_{1-\varphi}f(x),
	\end{split}
	\end{equation}
where $\mathcal M$ denotes the Hardy--Littlewood maximal operator on $\R$. Thus, from \cite[Theorem 5.6.6]{grafakos} we obtain
	\begin{equation*}
	\norm{ \left( \sum_{s=1}^S \abs{H^*_{\varphi} k_s}^2 \right)^{1/2}}_{L^p(\R)} \lesssim p^*\norm{ \left( \sum_{s=1}^S \abs{k_s}^2 \right)^{1/2}}_{L^p(\R)}.
\end{equation*}
Hence, \eqref{eq:Hmvod} will be justified if we show that
\begin{equation}
	\label{eq:Hmvod'} 
	\norm{ \left( \sum_{s=1}^S \abs{H^*_{1-\varphi} k_s}^2 \right)^{1/2}}_{L^p(\R)} \lesssim p^*\norm{ \left( \sum_{s=1}^S \abs{k_s}^2 \right)^{1/2}}_{L^p(\R)}.
\end{equation}

We will obtain \eqref{eq:Hmvod'} from \cite[Theorem 5.6.1]{grafakos} applied to
	$$
	\mathcal{B}_1 = \ell^2\left( \{1, \dots, S\} \right)\qquad\textrm{and} \qquad\mathcal{B}_2 = \ell^2\left( \{1, \dots, S\}; L^\infty{(0,\infty)} \right)
	$$ and
\begin{equation} \label{eq:kernelK}
	\vec{K}(x)(u) = \left( (1-\varphi_{t}) \chi_{t}(x) \cdot u_1 ,\ldots, (1-\varphi_{t}) \chi_{t}(x) \cdot u_S\right) \in \mathcal{B}_2,
\end{equation}
for any sequence $u=(u_s)_{s=1}^S \in \mathcal{B}_1$. Then, taking $e_s = (0, \dots, 1, \dots, 0)$, with $1$ on the $s$-th coordinate, we can see that the operator $\vec{T}$ defined in \cite[5.6.4]{grafakos} satisfies
\begin{equation} \label{eq:def_T}
\vec{T}\left( \sum_{s=1}^S f_s e_s \right)(x) = \left(H^{t}_{1-\varphi} f_1(x),\ldots,H^{t}_{1-\varphi} f_S(x)\right)
\end{equation}
and
\[
	\norm{\vec{T}\left( \sum_{s=1}^S f_s  e_s \right)(x)}_{{\mathcal{B}_2}} = \left(\sum_{s=1}^S \abs{H^*_{1-\varphi} f_s(x)}^2 \right)^{1/2};
\]
for any sequence $(f_s)_{s=1}^S$ of smooth functions that vanish at infinity.
In order to use \cite[Theorem 5.6.1]{grafakos} we need to verify conditions (5.6.1), (5.6.2) and (5.6.3) from \cite{grafakos} and check that $\vec{T}$ is bounded from $L^2(\R, \mathcal{B}_1)$ to $L^2(\R, \mathcal{B}_2)$.

The condition (5.6.1) is a straightforward consequence of \eqref{eq:kernelK} and the fact that $\vec{K}$ is an odd function implies that the condition (5.6.3) is also satisfied with $\vec{K}_0=0$.

To check the condition (5.6.2) denote $g_t(x) = (1-\varphi_{t}) \chi_t(x)$ and observe that
\[
	\norm{\vec{K}(x-y) - \vec{K}(x)}_{\mathcal{B}_1 \to \mathcal{B}_2} = \sup_{t>0} \abs{g_t(x-y) - g_t(x)}.
\]
In view of the above equality and the fact that $g_t(x)=(1-\varphi_{t}(x))x^{-1}$ a short calculation shows that
	$$	\norm{\vec{K}(x-y) - \vec{K}(x)}_{\mathcal{B}_1 \to \mathcal{B}_2}\lesssim\frac{|y|}{|x-y|^2},\qquad |x|\geqslant 2|y|,$$
	so that (5.6.2) follows.

It remains to justify the boundedness of $\vec{T}$ from $L^2(\R, \mathcal{B}_1)$ to $L^2(\R, \mathcal{B}_2).$ A reasoning similar to \eqref{eq:H*split} gives the pointwise bound
\[
 H^*_{1-\varphi}f(x)\lesssim \mathcal M f(x)+H^* f(x).
\]
Therefore the desired $L^2$ boundedness of $\vec{T}$ is  a consequence of \eqref{eq:def_T} and the $L^2(\R)$ boundedness of $\mathcal M$ and $H^*$. This allows us to use \cite[Theorem 5.6.1]{grafakos} and completes the proof of \eqref{eq:Hmvod'} hence also the proof of \thref{thm1''}.

%The estimate \eqref{eq:Hmvod'} for $1<p\leqslant 2$ follows by an application of \cite[Theorem 5.6.1]{grafakos}
%with $B_2=\ell^2(\ell^\infty)$ and $B_1=\ell^2$.
%For $2\leqslant p <\infty$ we pass by duality to the case $1<p\leqslant 2$ and apply the theorem
%5.6.1 with $B_1=\ell^2(\ell^1)$ and $B_2=\ell^2$. Verifying the conditions 5.6.1 and 5.6.3
%is immediate, the conditon 5.6.2 is standard 

%This completes the proof of \eqref{eq:Hmvod} hence also the proof of \thref{thm1''}.
%
%	We begin with the left-hand side of \eqref{eq35} using Cotlar's inequality, dimension-free estimates for the vector-valued Hardy--Littlewood maximal operator $\mathcal{M}$ (see \cite[Theorem 1]{deleaval_kriegler}), \cite[Proposition 3.1]{iwaniec_martin} and the fact that $\norm{H_\omega}_{L^p(\R^d)} = \norm{H}_{L^p(\R)}$ and we get
%	\begin{align*}
%		\norm{ \left( \sum_{s=1}^S \abs{H_\omega^* k_s}^2 \right)^{1/2}}_p &\lesssim \norm{ \left( \sum_{s=1}^S \abs{\mathcal{M} k_s}^2 \right)^{1/2}}_p + \norm{ \left( \sum_{s=1}^S \abs{\mathcal{M} H_\omega k_s}^2 \right)^{1/2}}_p \\
%		&\lesssim \norm{ \left( \sum_{s=1}^S \abs{k_s}^2 \right)^{1/2}}_p + \norm{ \left( \sum_{s=1}^S \abs{H_\omega k_s}^2 \right)^{1/2}}_p \\
%		&\lesssim \norm{ \left( \sum_{s=1}^S \abs{k_s}^2 \right)^{1/2}}_p.
%	\end{align*}
%	This proves \eqref{eq35} and completes the proof of the theorem.
\end{proof}

\end{document}